\documentclass[onefignum,onetabnum]{siamart171218}
\usepackage{lipsum}
\usepackage{amsfonts}
\usepackage{graphicx}
\usepackage{epstopdf}
\usepackage{algorithmic}
\ifpdf
  \DeclareGraphicsExtensions{.eps,.pdf,.png,.jpg}
\else
  \DeclareGraphicsExtensions{.eps}
\fi

\newsiamremark{remark}{Remark}
\newsiamremark{hypothesis}{Hypothesis}
\crefname{hypothesis}{Hypothesis}{Hypotheses}
\newsiamthm{claim}{Claim}

\headers{}{M. Hossein Gorji}

\title{Logarithmic Gradient Transformation and Chaos Expansion of It\^{o} Processes \thanks{Submitted to the editors \today.
\funding{This work was funded by Swiss National Science Foundation under contract no.~174060.}}}

\author{M. Hossein Gorji\thanks{MCSS, Department of Mathematics, EPF Lausanne 
  (\email{mohammadhossein.gorji@epfl.ch}).}}

\usepackage{amsopn}

\ifpdf
\hypersetup{
  pdftitle={Logarithmic Gradient Transformation and Chaos Expansion of It\^{o} Processes},
  pdfauthor={M. Hossein Gorji}
}
\fi

\begin{document}

\maketitle

\begin{abstract}
\noindent Since the seminal work of Wiener \cite{wiener}, the chaos expansion has evolved to a powerful methodology for studying a broad range of stochastic differential equations. Yet its complexity for systems subject to the white noise remains significant. The issue appears due to the fact that the random increments generated by the Brownian motion, result in a growing set of random variables with respect to which the process could be measured. In order to cope with this high dimensionality, we present a novel transformation of stochastic processes driven by the white noise. In particular, we show that under suitable assumptions, the diffusion arising from white noise can be cast into a logarithmic gradient induced by the measure of the process. Through this transformation, the resulting equation describes a stochastic process whose randomness depends only upon the initial condition. Therefore the stochasticity of the transformed system lives in the initial condition and thereby it can be treated conveniently with the chaos expansion tools. 
\end{abstract}



\section{Introduction}
\noindent Often stochastic descriptions of natural or social phenomena lead to more realistic mathematical models. The introduced stochastic notion may either arise from the uncertainty in the model inputs, or from the underlying governing law. In particular, the white noise manifests itself in both circumstances e.g. as a random force acting on a deterministic system in the Landau-Lifschitz fluctuating hydrodynamics \cite{landau} or as a Markovian process describing rarefied gases \cite{gorji} or polymers \cite{ottinger}.   \\ \ \\
The Monte--Carlo methods are typically a natural choice for computational studies of the systems driven by the white noise. Yet the slow convergence rate of the brute-forth Monte--Carlo, motivates a quest for improved approaches. There exists an immense list of advanced Monte--Carlo techniques, each of which may yield to a substantial improvement over the conventional Monte--Carlo, provided certain regularities. One of the promising examples belongs to the Multi-Level Monte-Carlo approach \cite{giles} (and its variants \cite{tempone}). In short, MLMC makes use of abundant samples on a coarse scale discretization in order to improve the convergence rate of the fine scale one. This can be achieved by enforcing correlations between successive approximations; usually through employing common random numbers among them. \\ \ \\
Instead of producing numerical samples of a random variable however, one can expand the solution with respect to a set of (orthogonal) random functions which possess a known distribution \cite{karniadakis1}. The polynomial chaos and stochastic collocation schemes are among the main approaches built around this idea \cite{karniadakis2,hesthaven}. In particular, the polynomial chaos schemes transform the random differential equations to a set of deterministic equations, through which the evolution of the coefficients introduced in the polynomial expansion of the random solution is governed. Therefore by knowing the distribution of the resulting orthogonal functions, different statistics of the solution can be computed deterministically. While this approach may lead to efficient computations for equations pertaining a finite set of random variables, its application to the Brownian motion remains with a significant computational challenge. The problem arises due to the fact that the dimension of the expansion should grow in time in order to keep the solution measurable with respect to the Brownian motion \cite{hou}. Hence, the cost of the chaos expansion schemes grows here significantly, in comparison to the counterpart scenario where the solution remains measurable with respect to a fixed set of random variables.  \\ \ \\
This paper addresses the problem of deterministic solution algorithms for systems subject to the white noise, in an idealized It\^{o} process setting. Here we introduce a novel transformation, where the randomness of the Brownian motion is described as a propagation of an (artificial) uncertainty of the initial condition. We show that the measure induced by the transformed system is consistent with the one resulting from the It\^{o} process, in the moment sense. The key ingredient is the fact that both the transformed and the original process result in an identical Fokker--Planck equation for their probability densities. Afterwards, since the transformed system describes an Ordinary Differential Equation (ODE) with an uncertain initial condition, a chaos expansion can be applied in a straight-forward manner.  \\ \ \\
The paper is structured as the following. First in the next section we present our setting for the It\^{o} process and besides a shoer review of its corresponding Wiener-chaos expansion. In \cref{sec:main}, the gradient transformation of the white noise is motivated and introduced. In the follow up \cref{sec:theory}, some theoretical aspects of the transformation are justified. In particular, the solution existence and uniqueness of the transformed process is discussed. Therefore in \cref{sec:chaos}, the Hermite chaos expansion of the transformed process is devised. The paper concludes with final remarks and future outlooks.

\section{Review of the Ito Process}
\label{sec:rev}
To start, a set of assumptions on the coefficients of the It\^{o} process, necessary for our analysis is provided in \cref{sec:gen}. Next, the conventional chaos expansion of the It\^{o} process is reviewed in \cref{sec:conv}.
\subsection{General Setting}
\label{sec:gen}
\noindent We focus on a simple prototype of stochastic processes driven by the white noise. Let $(\Omega,\mathcal{F}^{U_0}_{t},\mathcal{P})$ be a complete probability space, where $\mathcal{F}^{U_0}_t=\mathcal{F}_t \otimes \mathcal{F}^{U_0}$ denotes the $\sigma$-algebra on the subsets of $\Omega=\Omega_1 \cup \Omega_2$. Here $\{ \mathcal{F}_t\}_{t\ge 0}$ is an increasing family of $\sigma$-algebras induced by the $n$-dimensional standard Brownian path $W(.,.): \mathbb{R}^{+} \times \Omega_1\to \mathbb{R}^n$ , and $\mathcal{F}^{U_0}$ the $\sigma$-algebra generated by the initial condition $U_0(.): \Omega_2 \to \mathbb{R}^n$. \\ \ \\
We consider an It\^{o} diffusion process
\begin{equation}
dU_i (t,\omega)=b_i(U)  dt+\beta dW_i (t,\omega) ,
\label{eq:ito-main}
\end{equation}
governing the evolution of the $\mathcal{F}^{U_0}_{t}$-measurable random variable $U(.,.): \mathbb{R}^+\times \Omega \to \mathbb{R}^{n}$, with the initial value $U_0$ and the law $\mathcal{P}$. \\ \ \\
Throughout this manuscript, we need certain regularity assumptions on the drift $b(.) : \mathbb{R}^n \to \mathbb{R}^n$, the diffusion coefficient $\beta\in \mathbb{R}$ and the initial condition $U_0$. \\ \ \\
We require $\beta\neq 0$ and that the drift $b(x)=-\nabla \Psi(x)$ with $\Psi(.) \in{{C}^\infty_{b}}(\mathbb{R}^n)$, where $C_{b}^{\infty}$ denotes the space of bounded functions with bounded derivative of all orders. 
Finally, we assume that the initial condition is deterministic hence its probability density $f_{U_0}(u)=\delta (u-U_0)$, where $\delta(.)$ is the n-dimensional Dirac delta and $U_0\in \mathbb{R}^n$. \\ \ \\
For the above-described setting, many interesting properties can be shown for the It\^{o} process, including the following.
\begin{remark}
\label{rem:sde}
 It is a classic result in the theory of Stochastic Differential Equations (SDEs) that since $\Psi(.) \in{{C}^\infty_{b}}(\mathbb{R}^n)$ and $\beta$ is assumed to be a constant, Eq.~\cref{eq:ito-main} has a solution with a bounded variance for all $t \ge 0$, which is unique in the mean square sense. Furthermore, the process is Feller continuous resulting in smooth variation of an expectation of the solution with respect to the initial condition \cite{oksendal}. 
 \end{remark}
\begin{remark}
\label{rem:cont}
Based on different results in the Malliavin calculus, since the coefficients $b$ and $\beta$ fulfill the Horm\"{a}nder criterion and furthermore $b$ has bounded derivatives, the Borel measure generated by the process $\mu_{U}=\mathcal{P}(U^{-1})$ is infinite times differentiable. Therefore the probability density $f_U(u;t)du=d\mu_U(u;t)$ is well-defined and $\mu_U(.;t),f_U(.;t)\in C^\infty(\mathbb{R}^n)$, provided $t>0$; see e.g. Theorem 2.7 in \cite{watanabe}. 
\end{remark}
\begin{remark}
\label{rem:fisher}
Due to Corollary 4.2.2. of \cite{bogachev}, since $\mu_U$ is three times differentiable, the Fisher information 
\begin{eqnarray}
I(f)&:=&\int_{\mathbb{R}^n}\frac{1}{f}\nabla_x f \cdot \nabla_x f dx
\end{eqnarray}
associated with the density $f_U$ is bounded at $t>0$.
\end{remark}
\begin{remark}
\label{rem:fp}
The density $f_U$ evolves according to the Fokker-Planck equation (forward-Kolmogorov equation)
\begin{eqnarray}
\label{eq:fp}
\frac{\partial f_U(u;t)}{\partial t}&=&-\frac{\partial}{\partial u_i}\left(b_i (u) f_U(u;t)\right)+\frac{\beta^2}{2}\frac{\partial^2}{\partial u_i \partial u_i }f_U(u;t)
\end{eqnarray}
and the measure $\mu_U$ is governed by the transport equation
\begin{eqnarray}
\label{eq:trans}
\frac{\partial \mu_U(u;t)}{\partial t}&=&-b_i (u) \frac{\partial}{\partial u_i}\mu_U(u;t)+\frac{\beta^2}{2}\frac{\partial}{\partial u_i\partial u_i }\mu_U(u;t).
\end{eqnarray}
Since $\Psi(.) \in{{C}^\infty_{b}}(\mathbb{R}^n)$ and $\beta \neq 0$, both above-mentioned equations have unique solutions (for uniqueness results see \cite{manita,diperna,bogachev2}). Notice that the Einstein index convention is employed here and henceforth, to economize the notation.
\end{remark}
\noindent In comparison to the natural setting of It\^{o} processes, we have introduced strong assumptions on $\Psi$ and $\beta$. Though not straight-forward, the generalization of our analysis may become possible as long as the corresponding It\^{o} process has a unique solution with bounded variance and its corresponding Fisher information is bounded (e.g. by using Lyapunov functionals \cite{khasminskii}). But to keep the study focused on the main idea, we postpone the generalization to the follow up studies. \\ \ \\ 
In typical applications in scientific computations, one is interested in some moments of the solution $U$, which are in the form of an expectation $\mathbb{E}[g (U(t,\omega))]$ of some smooth function $g(.)\in C^{\infty}(\mathbb{R}^n)$. 
\subsection{Wiener Chaos Expansion}
\label{sec:conv}
\noindent Due to slow convergence rates of Monte-Carlo methods, deterministic solution algorithms for stochastic processes can be attractive. Besides stochastic collocation methods \cite{zhang}, a Wiener chaos expansion of Eq.~\eqref{eq:ito-main} is possible due to the Cameron-Martin theorem \cite{cameron}, as carried out e.g. by Rozovskii, Hou and others \cite{hou, rozovskii,karniadakis1}. It is useful for our sequel analysis to provide an overview of this expansion. 
To simplify the notation we explain the chaos expansion of $U$ in a one dimensional setting $n=1$. For a multi-dimensional case, the following can be applied for each component of the solution. \\ \ \\
The random events with respect to which the solution $U$ is measurable are due to the initial condition $U_0$ and the corresponding Brownian integral $\beta\int_0^t dW(s,\omega_2)$, therefore for a deterministic $U_0$, $U$ can be expressed as 
\begin{eqnarray}
U(t,\omega)&=&M\left(U_0, \int_0^t dW(s,\omega),t\right). 
\end{eqnarray}
The integral of the Brownian path $\mathcal{I}(\omega):= \int_{s=0}^{t} dW(s,\omega)$, can be expanded as
\begin{eqnarray}
\mathcal{I}(\omega)&=&\sum_{j=1}^\infty\xi_j(\omega) \int_{0}^t\phi_j(s)ds,
\label{eq:exact-wiener}
\end{eqnarray}   
where $\{\phi_j(s)\}$ is a sequence of orthogonal functions in $L^2([0,t])$ and $\xi_j$ are independent normally distributed random variables. \\ \ \\
Suppose $P^{(l)}=\left\{t^{(l)}_j=\bigg(jt/{m_l} \ \ j\in\{1,...,m_l\} \bigg) \right\}$ is a partition for the time interval $(0,t]$. Intuitively the Brownian motion generates an independent normally distributed random variable at each $t^{(l)}_j\in P^{(l)}$. Along this picture let 
\begin{eqnarray}
\hat{\mathcal{I}}^{(l)}&=&\sum_{j=1}^{m_l}\xi_j\int_{0}^t\phi_j(s)ds
\label{eq:approx-wiener}
\end{eqnarray}
 be an approximation of the integral \eqref{eq:exact-wiener} corresponding to the partition $P^{(l)}$. It can be shown that 
 \begin{eqnarray}
\mathbb{E}\left[\left(\mathcal{I}-\hat{\mathcal{I}}^{(l)}\right)^2\right]< C \frac{t}{m_l},
 \end{eqnarray}
 where $C<\infty$ is some constant  \cite{luo}.\\ \ \\
 Analogously, let $\hat{U}^{(l)}$ be an approximation of $M$, computed on the partition $P^{(l)}$. Therefore due to Eq.~\eqref{eq:approx-wiener}, the solution at time $t$ can be approximated as a function $\hat{U}^{(l)}(t,\xi_1,...,\xi_{m_l})$ with a mean square error of $\mathcal{O}(t/m_l)$ (due to the truncation introduced in Eq.~\eqref{eq:approx-wiener}). At this point the Wiener chaos expansion can be applied to $\hat{U}^{(l)}$; as explained in the following. \\ \ \\
In order to expand $\hat{U}^{(l)}$ with respect to the Hermite basis, suppose $\xi=(\xi_1,...,\xi_{m_l})$ is an $m_l$-dimensional normally distributed random variable and let
 $\alpha=(\alpha_1,...,\alpha_p)\in \mathcal{J}^p_{m_l}$ denote an index from the set of multi-indices 
\begin{eqnarray}
\label{eq:multi}
\mathcal{J}^p_{m_l}&=&\left\{\alpha=(\alpha_i,1\le i\le {m_l}) \bigg \vert \alpha_i\in\{0,1,2,...,p\},|\alpha|=\sum_{i=1}^{m_l}\alpha _i\right\}.
\end{eqnarray}
Let the $|\alpha|$-order multi-variate Hermite polynomial
\begin{eqnarray}
{H}_\alpha(\xi)&=&\prod_{i=1}^{m_l}\hat{H}_{\alpha_i}(\xi_i)
\label{eq:def-wick}
\end{eqnarray}
be a tensor product of the normalized $\alpha_i$-order Hermite polynomials $\hat{H}_{\alpha_i}(\xi_i)$. 
According to the Cameron-Martin theorem, $\hat{U}^{(l)}(t,\xi)$ admits the following Hermite expansion
\begin{eqnarray}
\hat{U}^{(l)}(t,\xi)=\lim_{p \to \infty}\sum_{\alpha \in \mathcal{J}^p_{m_l}} {\hat{u}^{(l)}}_{\alpha}(t) H_\alpha(\xi),
\label{eq:strong-Hermite}
\end{eqnarray}
where ${\hat{u}^{(l)}}_\alpha(t)=\mathbb{E}[\hat{U}^{(l)}(t,\xi) H_\alpha(\xi)]$. \\ \ \\
In fact the expansion \eqref{eq:strong-Hermite} provides a means to project the randomness of the solution $U(t,\omega)$ into the Hermite basis. As a result, the It\^{o} process is transformed to a set of deterministic ODEs for the coefficients $\hat{u}^{(l)}_{\alpha}(t)$ and thus the expectations $\mathbb{E}[g(U(t,\omega)]\approx\mathbb{E}[g(\hat{U}^{(l)}(t,\xi))]$ can be computed deterministically. However in order to keep the order of the approximation introduced in the expansion \eqref{eq:approx-wiener} constant, $m_l$ should grow linearly with respect to $t$. So does the dimension of the expansion \eqref{eq:strong-Hermite}, as $m_l$ shows up in the order of the Hermite polynomials.  
Thus unless short time behavior of the solution is of interest, complexity of the Wiener chaos expansion of the It\^{o} process may become prohibitive; even though the number of Hermite polynomials can be reduced significantly through sparse tensor compressions \cite{schwab}.  \\ \ \\
A more general insight about the problem can be sought by considering the fact that a smooth function of an $n$-dimensional random process Brownian path $f(W(t,\omega))$ at time $t=T$ is measurable with respect to the Borel $\sigma$-algebra on $\Omega={\left(\mathbb{R}^{n}\right)}^{[0,T]}$  \cite{oksendal}. Therefore in order to devise a chaos expansion of $f$, the orthogonal functions should span a rather high dimensional space $L^2(\Omega)$. 
\section{Main Result}
\label{sec:main}
\noindent The main idea of this work is to find an alternative SDE with a similar probability density as the one generated by the It\^{o} process, which yet remains measurable with respect to the $\sigma$-algebra induced by its initial condition. \\ \ \\
More precisely, consider again the partition $P^l=\{0=t^l_1<t_2^l<...<t^l_{m_l}=t\}$ for the time interval $[0,t]$ with $|P^l|\to 0$ as $l\to \infty$. Obviously the solution of the It\^{o} process $U(t,\omega)$ is measurable with respect to the family of $\sigma$-algebras  
\begin{eqnarray}
\{\mathcal{F}^{U_0}_{t^l_1}, \mathcal{F}^{U_0}_{t^l_2}...,\mathcal{F}^{U_0}_{t^l_{m_l}} \} \ \ \ \textrm{as} \ \ \ l\to \infty. \nonumber
\end{eqnarray}
However if we are only interested in some expectation $\mathbb{E}[g(U(t,\omega))]$ at time $t$, the knowledge of the Borel measure $\mu_U(B;t)=\mathcal{P}\{U^{-1} (t,B)\}$ where $B\in \mathcal{B}^n$, is sufficient. Note that $\mathcal{B}^n$ is the Borel $\sigma$-algebra on $\mathbb{R}^n$. Let $f_U(u;t)$ be the corresponding probability density i.e. $f_U(u;t)du=d\mu_U(u;t)$, therefore 
\begin{eqnarray}
\mathbb{E}[g(U(t,\omega))]&=&\int_{\mathbb{R}^n}f_U(u;t)g(u)du. \nonumber
\end{eqnarray}   
\noindent Suppose the random variable $X(t,\omega): \mathbb{R}^+ \times \Omega \to \mathbb{R}^n$ belongs to a complete probability space $(\Omega,\mathcal{G},\mathcal{Q})$, and generates a Borel measure $\mu_X=\mathcal{Q}(X^{-1})$. Let the probability density be $f_X(x;t)dx=d\mu_X$. 
We propose that under suitable assumptions on $f_{X}(x;0)$ (as explained in the following section), the solution of the transformed It\^{o} process
\begin{eqnarray}
\frac{d}{dt}X_i(t,\omega)&=&b_i(X)-\frac{1}{2} \beta^2 \left [\nabla_{x_i} \log f_X(x;t) \right ]_{x=X(t,\omega)}
\label{eq:weak-ito}
\end{eqnarray}
with the initial condition $X_0(\omega):\Omega \to \mathbb{R}^n$, uniquely exists for all $t$. Furthermore the solution is consistent with the It\^{o} process in a sense that for an arbitrary smooth $g\in C^\infty (\mathbb{R}^n)$ we have
 \begin{eqnarray}
\mathbb{E}[g(X(\omega,t))]&=&\mathbb{E}[g(U(\omega,t))],
 \end{eqnarray}
where $U$ is the solution of the It\^{o} process with the initial condition $U_0=X_0$. \\ \ \\
Let us first review the motivation behind this transformation. Due to It\^{o}'s lemma, the probability density generated by the It\^{o} process follows the Fokker-Planck equation (see \cref{rem:fp})
\begin{eqnarray}
\frac{\partial f_U(u;t)}{\partial t}+\frac{\partial}{\partial u_i}\left(b_i (u) f_U(u;t)\right)&=&\frac{1}{2}\frac{\partial^2}{\partial u_i \partial u_j }\left(\beta^2f_U(u;t)\right).
\end{eqnarray}
By rearranging the diffusion term one can see that
\begin{eqnarray}
\frac{\partial f_U(u;t)}{\partial t}+\frac{\partial}{\partial u_i}\left\{\bigg(b_i(u)-\frac{1}{2}\beta^2 \frac{\partial}{\partial u_j} \log(f_U(u;t))\bigg)f_U(u;t)\right\}&=&0, \nonumber
\end{eqnarray}
resulting in a stochastic process similar to Eq.~\eqref{eq:weak-ito}. Intuitively we observe that the effect of the diffusion on the probability density is equivalent to an advection induced by the gradient $\nabla_u \log f_U$. We refer to this transformation as {\it logarithmic gradient transformation}. 
Obviously this transformation needs to be justified. However before proceeding to the technical discussion in \cref{sec:theory}, let us provide some physical motivations behind the logarithmic gradient transformation. \\ 
\noindent Suppose   $\exp\left({-{2\Psi(x)}/{\beta^2}}\right)\in L^1(\mathbb{R}^n)$ and hence the stationary density
\begin{eqnarray}
f_{st}(x)&=&\mathcal{Z}\exp\left({-\frac{2\Psi(x)}{\beta^2}}\right)  
\end{eqnarray}
is well-defined.
Therefore the introduced process generates the paths $(t,X(t,\omega))$ according to
\begin{eqnarray}
\frac{d}{dt}X_i(\omega,t)&=&-\frac{\beta^2}{2}\nabla_x \log \left(\frac{f_X(x;t)}{f_{st}(x)}\right) \bigg\vert_{x=X(\omega,t)} \nonumber 
\end{eqnarray}
which is a gradient flow induced by the potential $\phi=\log ({f_X}/{f_{st}}) $. This potential is connected to the Kullback-Leibler distance (entropy distance)
\begin{eqnarray}
d_{KL}(t)&=&\int_{\mathbb{R}^n}f_X(x;t) \log \left(\frac{f_X(x;t)}{f_{st}(x)}\right) dx=\mathbb{E}[\phi(X)] \nonumber 
\end{eqnarray}
between the two densities $f_X$ and $f_{st}$ \cite{kullback,otto}.
Therefore from the physical point of view, the logarithmic gradient transformation generates a gradient flow in order to minimize the entropy distance $d_{KL}$ between the current state $f_X$ and $f_{st}$. \\ \ \\
\section{Theoretical Justifications}
\label{sec:theory}
\noindent The following arguments establish a connection between solutions of the main It\^{o} process i.e. Eq.~\eqref{eq:ito-main} and the transformed one Eq.~\eqref{eq:weak-ito}. 
\subsection{Regularity of the Ito Process}
\noindent To start, note that in order to make sense of Eq.~\eqref{eq:weak-ito}, $f_U$ should admit certain regularities. 
Let us introduce a class of admissible probability densities for a measurable $f(x)$ as
\begin{eqnarray}
K_1&:=&\bigg\{ f(x) : \mathbb{R}^n\to (0,\infty)\  \bigg \vert \ \nabla \log{f}\in C_l^{\infty}(\mathbb{R}^n),\ M(f)<\infty,I(f)<\infty \bigg\},
\end{eqnarray}
where
\begin{eqnarray}
M(f)&=&\int_{\mathbb{R}^n}fx^2dx \nonumber  
\end{eqnarray}
and $C_l^{\infty}$ is the space of infinite times differentiable functions, with at most linear growth.
\noindent The next lemma provides a link between $f_U$ and $K_1$. 
\begin{lemma} \label{lemma:space}
Consider $U^\epsilon(t,\omega)$ to be the solution of the It\^{o} process \eqref{eq:ito-main} in the probability space $(\Omega,\mathcal{F}_t^{U^\epsilon_0},\mathcal{P}^\epsilon)$ with a drift $b=-\nabla \Psi$, $\Psi(.) \in C_b^\infty(\mathbb{R}^n)$ and a diffusion $\beta \neq 0$. Suppose the initial condition reads $U_0^\epsilon=U_0+\epsilon Z$, where $U_0\in \mathbb{R}^n$ is deterministic, $Z(\omega)\in \mathbb{R}^n$ is a normally distributed random variable and $\epsilon \in \mathbb{R}$ is a small, arbitrary chosen non-zero constant. \\ \ \\
Let $f_{U^\epsilon}(u;t)=d\mathcal{P}^\epsilon\left({U^\epsilon}^{-1}\right)$ be the probability density of the process, therefore \\
\begin{eqnarray}
\label{in:reg-ito}
f_{U^\epsilon}(.;t)\in K_1,
\end{eqnarray}
for $t\in [0, \infty)$.
\end{lemma}
\begin{proof} 
Note that the initial condition $U_0^\epsilon$ has a Gaussian probability density of the form
\begin{eqnarray}
f_{U_0^\epsilon}(u)=\mathcal{M}_{\epsilon}\left(|u-U_0|\right),
\end{eqnarray}
where
\begin{eqnarray}
\label{eq:gauss}
\mathcal{M}_\epsilon (h)&:=&\frac{1}{(\sqrt{2\pi}|\epsilon| )^{n}}\exp\left(-\frac{h^2}{2\epsilon^2}\right).
\end{eqnarray}
It is straight-forward to see that $\mathcal{M}_{\epsilon}\left(|u-U_0|\right) \in K_1$ and thus we only need to prove the claim \eqref{in:reg-ito} for $t>0$. Notice that here and afterwards, $|\ . \ |$ denotes the Euclidean norm. \\ \ \\
First let us show that $\log f_{U_0^\epsilon}(.;t>0)\in C^{\infty}(\mathbb{R})$. According to \cref{rem:sde}-\cref{rem:fisher} at each $t>0$ we have $f_{U_0^\epsilon}(.;t)\in C^{\infty}(\mathbb{R})$, $I(f_{U_0^\epsilon})<\infty$ and $M(f_{U_0^\epsilon})<\infty$. Hence it is sufficient to prove $f_{U_0^\epsilon}(.; t)>0$, for $t>0$. For that, we make use of the Girsanov transformation. But before proceed, to prevent unnecessary notational complications we set $\beta =1$ for the followings. \\ \ \\
Let $W^\epsilon(t,\omega)$ be a standard n-dimensional Brownian process with the initial condition $U^\epsilon_0$ and the law $\mathcal{W}^\epsilon$. Then since $b(.)\in C_{b}^\infty(\mathbb{R}^n)$, we have
\begin{eqnarray}
\mathbb{E}\left[\exp\left(\frac{1}{2}\int_0^{T} b_i (W^\epsilon(t,\omega))b_i(W^\epsilon(t,\omega)dt \right)\right]&<&\infty,
\end{eqnarray}
for any finite $T$.
 Therefore the process
\begin{eqnarray}
Z(t,\omega)&:=&\exp\left(-\int_0^t b_i(W^\epsilon(s,\omega))dW^\epsilon_i(s,\omega)-\frac{1}{2}\int_{0}^t b^2(W^\epsilon(s,\omega))ds \right) \nonumber \\
\end{eqnarray}
is a martingale for $t\in [0,T)$ \cite{oksendal}. It follows from the Girsanov theorem that
\begin{eqnarray}
d\mathcal{P}^\epsilon(t,\omega)&=&Z(t,\omega)d\mathcal {W}^\epsilon(t,\omega).
\end{eqnarray}
Since $d\mathcal{W}^\epsilon$ is a Gaussian measure, it is strictly positive for $t>0$, and hence $d\mathcal{P}>0$. It is then straight-forward to check that $f_{U_0^\epsilon}(u;t)>0$, for any $u\in \mathbb{R}^n$, provided $t>0$. \\ \ \\
Now the final piece is to prove 
\begin{eqnarray}
|\nabla_u \log f_{U^\epsilon}(u;t)| &\leq& C(t,U_0)\left(|u|+1\right)
\end{eqnarray}
for every $u\in \mathbb{R}^n$, $t>0$ and some constant $C(t,U_0)<\infty$ which depends on $t$ and the initial condition $U_0$. Consider the partition 
\begin{eqnarray}
P^{(l)}&=&\left\{t^{(l)}_j=\bigg(jt/{m_l} \ \ j\in\{1,...,m_l\} \bigg) \right\}
\end{eqnarray}
for the interval $(0,t]$ and $\Delta t^{(l)}=t/{m_l}$. Suppose ${Z}^{(l)}$ is the projection of the martingale $Z(t,\omega)$ on the partition $P^{(l)}$. Using It\^{o}'s lemma, we get
\begin{eqnarray}
Z^{(l)}(t,\omega)&=&\exp\bigg(\Psi(W^\epsilon(0,\omega))-\Psi(W^\epsilon(t,\omega))\bigg) \nonumber \\ 
&&\exp\left(\frac{1}{2}\sum_{j=1}^{m_l}\left({b^\prime}(W^\epsilon(t_j^{(l)},\omega)-b^2(W^\epsilon(t_j^{(l)},\omega))\right)\Delta t^{(l)}\right), \nonumber \\
\end{eqnarray}
where ${b^\prime}=\textrm{div}\{b\}$. In terms of the density $f_{U^\epsilon}$, the Girsanov transformation  yields
\begin{eqnarray}
\label{eq:girsanov-f}
f_{U^\epsilon}(u_{m_l};t)&=&e^{-\Psi(u_{m_l})}\underbrace{\int_{\mathbb{R}^n}...\int_{\mathbb{R}^n}}_{m_l \ \textrm{times}}\bigg(e^{\Psi(u_0)+1/2\Delta t^{(l)}\sum_{j=0}^{m_l-1}\left(b^\prime(u_j)-b^2(u_j)\right)}\nonumber \\
&&\mathcal{M}_\epsilon (|u_0-U_0|)\prod_{i=0}^{m_l-1}\mathcal{M}_{\Delta t^{(l)}} (|u_{i+1}-u_i|)\bigg)du_0du_1...du_{m_l-1},
\end{eqnarray}
as $m_l\to \infty$, where $\mathcal{M}$ is the Gaussian density defined in Eq.~\eqref{eq:gauss}. Since $\Psi \in C^\infty_b$, $\exp(\Psi(u_0)+1/2\Delta t^{(l)}\sum_{j=0}^{m_l-1}\left(b^\prime(u_j)-b^2(u_j)\right)$ is bounded above and below by some $S(t) < \infty$ and $I(t)>0$, respectively. Therefore we have
\begin{eqnarray}
\bigg \vert \nabla_{u_{m_l}}&&\log f_{U^\epsilon}(u_{m_l};t) \bigg \vert\leq|b(u_{m_l})|\nonumber \\
+\frac{S(t)}{I(t)}&&\left\vert \frac{\int_{\mathbb{R}^n}...\int_{\mathbb{R}^n}\mathcal{M}_{\epsilon^2} (|u_0-U_0|)\prod_{i=0}^{m_l-1}\nabla_{u_{m_l}}\mathcal{M}_{\Delta t^{(l)}} (|u_{i+1}-u_i|)du_0...du_{m_l-1}}{\int_{\mathbb{R}^n}...\int_{\mathbb{R}^n}\mathcal{M}_{\epsilon^2} (|u_0-U_0|)\prod_{i=0}^{m_l-1}\mathcal{M}_{\Delta t^{(l)}} (|u_{i+1}-u_i|)du_0...du_{m_l-1}}\right \vert, \nonumber \\
\end{eqnarray}
as $m_l\to \infty$.
However, the integral terms can be computed explicitly. In fact in the limit of $m_l\to \infty$, we get
\begin{eqnarray}
\int_{\mathbb{R}^n}...\int_{\mathbb{R}^n}\mathcal{M}_{\epsilon^2} (|u_0-U_0|)\prod_{i=0}^{m_l-1}\mathcal{M}_{\Delta t^{(l)}} (|u_{i+1}-u_i|)du_0...du_{m_l-1}&=&\mathcal{M}_{\epsilon^2+t}(|u_{m_l}-U_0|). \nonumber \\
\end{eqnarray}
Therefore the upper bound reads
\begin{eqnarray}
\left \vert \nabla_{u_{m_l}} \log f_{U^\epsilon}(u_{m_l};t) \right \vert&\leq&|b(u_{m_l})|
+\frac{S(t)}{I(t)}\left\vert \frac{\nabla_{u_{m_l}}\mathcal{M}_{\epsilon^2+t}(|u_{m_l}-U_0|)}{\mathcal{M}_{\epsilon^2+t}(|u_{m_l}-U_0|)}\right \vert \nonumber \\
&&\leq C(t,u_0)\left(|u_{m_l}|+1\right),
\end{eqnarray}
for $t>0$.
\end{proof}
\begin{corollary}\label{cor:trans}
The measure of the process $\mu_{U^\epsilon}$ is the solution of the following transport equation
\begin{eqnarray}
\frac{\partial \mu_U(u;t)}{\partial t}&=&\left(-b_i (u)+\frac{\beta^2}{2}\frac{\partial }{\partial u_i}\log f_{U^\epsilon}(u;t)\right)\frac{\partial \mu_U(u;t)}{\partial u_i}.
\end{eqnarray}
\end{corollary}
\begin{proof}
The proof is straight-forward, by using \cref{rem:fp} and the result of \cref{lemma:space}, that $f_{U^\epsilon}(.,t) \in K_1$. 
\end{proof}
\subsection{Solution Existence-Uniqueness and Consistency}
\begin{theorem}
Let $U(t,\omega)$, $U^\epsilon(t,\omega)\in \mathbb{R}^n$ be solutions of the It\^{o} process \eqref{eq:ito-main} for initial conditions $U_0$ and $U_0^\epsilon$, respectively, where the drift $b=-\nabla \Psi$ fulfills $\Psi \in C_b^\infty$ and $\beta \neq 0$. Here $U_0\in \mathbb{R}^n$ is deterministic, whereas $U^\epsilon_0=U_0+\epsilon Z$, $Z(\omega)\in \mathbb{R}^n$ is a normally distributed random variable and $\epsilon \in \mathbb{R}$ is a non-zero arbitrary chosen parameter. \\ \ \\
Suppose $X^\epsilon(t,\omega)\in \mathbb{R}^n$ is a random variable in a space $(\Omega,\mathcal{G}^\epsilon,\mathcal{Q}^\epsilon)$, and evolves according to
\begin{eqnarray}
\label{eq:ito-weakp}
\frac{d}{dt}X^\epsilon_i(t,\omega)&=&b_i(X^\epsilon)-\frac{1}{2} \beta^2 \left [\nabla_{x_i} \log f_{X^\epsilon}(x;t) \right ]_{x=X^\epsilon(t,\omega)},
\end{eqnarray}
subject to the initial condition $U_0^\epsilon$. Here $f_{X^\epsilon}(x;t)=d\mathcal{Q}^\epsilon\left({X^\epsilon}^{-1}\right)$ is the probability density of the process \eqref{eq:ito-weakp}. Therefore
\begin{enumerate}
\item The process \eqref{eq:ito-weakp}, has a unique solution with $\mathbb{E}[{X^\epsilon}^2(t,\omega)]<\infty$ for $t\in[0,\infty)$.
\item For an arbitrary $g(.)\in C^2(\mathbb{R}^m)$, we have
\begin{eqnarray}
\mathbb{E}\left[g(X^\epsilon(t,\omega))\right]&=&\mathbb{E}\left[g(U^\epsilon(t,\omega))\right]  \\
\textrm{and} \ \ \ \ \ \lim_{\epsilon\to 0}\mathbb{E}\left[g(X^\epsilon(t,\omega))\right]&=&\mathbb{E}\left[g(U(t,\omega))\right].
\end{eqnarray}
\end{enumerate}
\end{theorem}
\begin{proof}
First let us show that the process 
\begin{eqnarray}
\frac{d}{dt}Y^\epsilon_i(t,\omega)&=&b_i(Y^\epsilon)-\frac{1}{2}\beta^2\left [\nabla_{y_i} \log f_{U^\epsilon}(y;t) \right ]_{y=Y^\epsilon(t,\omega)}
\label{eq:ode-weak}
\end{eqnarray}
with the initial condition $U_0^\epsilon$ has a unique solution with bounded variance for all $t>0$. Let $F(t,Y^\epsilon)$ denote the right hand side of Eq.~\eqref{eq:ode-weak}. For the existence-uniqueness proof of a bounded variance solution,  since $f_{U^\epsilon}(.;t)\in K_1$ according to \cref{lemma:space} and $b(.)\in C^\infty_b(\mathbb{R}^n)$, we get $F(t,.)\in C_l^\infty(\mathbb{R}^n)$.
Therefore the existence-uniqueness follows directly from the Picard iterations and Groenwall's inequality (see \cite{agarwal} for details). Furthermore, the boundedness of the variance comes from the Chebyshev lemma (see Theorem 1.8 in \cite{khasminskii}). \\ \ \\
Now let us turn to the measure induced by $Y^\epsilon$ i.e. $\mu_{Y^\epsilon}$. Let us define the map $\sigma_t(U_0^\epsilon(\omega))=Y^\epsilon (t,\omega)$ and hence $\mu_{Y^\epsilon}(\sigma_t(u);t)=\mu_{U^\epsilon_0}(u)$. Therefore $\mu_{Y^\epsilon}$ fullfills the following transport equation
\begin{eqnarray}
\label{eq:trans2}
\frac{\partial }{\partial t}\mu_{Y^\epsilon}(y;t)&=&-{F_i(t,y)}\frac{\partial }{\partial y_i} \mu_{Y^\epsilon}(y;t).
\end{eqnarray}
Note that since Eq.~\eqref{eq:ode-weak} has a unique solution, do does Eq.~\eqref{eq:trans2}. However due to \cref{cor:trans}, the measure induced by $U^\epsilon$ also fulfills Eq.~\eqref{eq:trans2}. Therefore $\mu_{Y^\epsilon}(y;t)=\mu_{U^\epsilon}(y;t)$, resulting in equivalence of Eqs \eqref{eq:ode-weak} and \eqref{eq:ito-weakp}. Furthermore 
\begin{eqnarray}
\mathbb{E}[g(X^\epsilon(\omega,t))]&=&\mathbb{E}[g(U^\epsilon(\omega,t))].
\end{eqnarray}
But since the It\^{o} process is Feller continuous \cite{oksendal}, we have
\begin{eqnarray}
\lim_{\epsilon\to 0}\mathbb{E}[g(U^\epsilon(\omega,t))]=\mathbb{E}[g(U(\omega,t))],
\end{eqnarray}
and hence
\begin{eqnarray}
\lim_{\epsilon \to 0}\mathbb{E}[g(X^\epsilon(\omega,t))]&=&\mathbb{E}[g(U(\omega,t))].
\end{eqnarray}
%
\end{proof}
\noindent To summarize, let $U^\epsilon$ and $U$ be solutions of the It\^{o} process subject to the initial conditions $U^\epsilon_0$ and $U_0$, respectively. As a consequence of the regularization and the introduced transformation, we can approximate the statistics of the true solution $U$ by statistic of $U^\epsilon$ through $\mathbb{E} [g(U^\epsilon(\omega,t))]=\mathbb{E}[g(X^\epsilon(\omega,t))]$. However due to well-posedness of Eq.~\eqref{eq:ito-main}, we obtain a mean square error 
\begin{eqnarray}
\mathbb{E}\left[(U(\omega,t)-U^\epsilon(\omega,t))^2\right] < C(t) \epsilon^2
\end{eqnarray}
bounded by $\epsilon^2$ and some constant $C(t)$ independent of $\epsilon$. Therefore the regularization costs us an error of $\mathcal{O}(\epsilon^2)$ in the mean square sense. \\ \ \\
\section{Chaos Expansion} \label{sec:chaos}
\noindent The computational advantage of the gradient formulation Eq.~\eqref{eq:weak-ito} over the original It\^{o} process Eq.~\eqref{eq:ito-main}, can be exploited through its chaos expansion. Actually while the dimension of the space in which the Brownian path is measurable increases in time, its gradient transformation only propagates randomness originated from the initial condition. Therefore the resulting logarithmic gradient transformation behaves like an ODE with an uncertain initial condition. \\ \ \\
Let us consider an initial condition $X_0(\omega):\Omega \to \mathbb{R}^n$ with a probability density $f_{X_0}(x)=\mathcal{M}_{\epsilon}(|x-U_0|)$, where $|\epsilon|>0$ and $U_0\in \mathbb{R}^n$. In the following, we present the corresponding Hermite chaos expansion of the process \eqref{eq:weak-ito} for $X(\omega,t):\Omega \times \mathbb{R}^+\to \mathbb{R}^n$ subject to $X_0$. For more details on the Hermite chaos, and in general polynomial chaos expansions see \cite{karniadakis2}. 
The expansion is performed on the map $M(\xi(\omega),t)=X(\omega,t)$, where $\xi\in \mathbb{R}^{n}$ is a normally distributed random variable, hence 
\begin{eqnarray}
| \nabla_{q} M |  f_{X}(M;t)&=&f_\Xi (q),
\label{eq:consist}
\end{eqnarray}
where $f_\Xi(q)=\mathcal{M}_1(q)$ and $q \in \mathbb{R}^{n}$. In practice, Eq.~\eqref{eq:consist} is only employed to find the initial condition of $M$ (which in our case of $X_0$ initially being Gaussian distributed, the map becomes trivial), afterwards simply the coefficients of the expanded $M$ are propagated.\\ \ \\
The map evolves according to $X$ and thus
\begin{eqnarray}
\frac{d}{dt}{M}_i(\xi(\omega),t)&=&\overbrace{b_i(M)-\frac{1}{2}\beta^2 \left[\nabla_{x_i}\log f_{{X}}(x;t) \right]_{M}}^{F_i(t,M)}.
\label{eq:weak3}
\end{eqnarray}
Since $\mathbb{E}[M^2]<\infty$, we conclude $M\in L^2(d\mu_\Xi)$, where $L^2(d\mu_\Xi)$ is the space of square integrable functions with the weight $d\mu_\Xi(q)=f_\Xi(q) dq$. Furthermore note that since $b(.)$ and the Fisher information are bounded, we have $F(t,.)\in L^2(d\mu_\Xi)$. Therefore $M$ admits a Hermite expansion \cite{sansoe}
 \begin{eqnarray}
{M}_{i}(\xi,t)&=&\lim_{p\to \infty}\sum_{\alpha \in \mathcal{J}_{n}^p} m_{i,\alpha}(t) {H}_{\alpha}(\xi)
\label{eq:exp-init}
\end{eqnarray}
for each component $i\in \{1,...,n\}$, where ${H}_\alpha$ and $\mathcal{J}$ are defined in \eqref{eq:def-wick} and \eqref{eq:multi}, respectively. The coefficients follow
\begin{eqnarray}
m_{i,\alpha}(t)&=&\left \langle M_{i} ,{H}_\alpha \right \rangle_{\mu_{\Xi}}, \label{eq:proj}
\end{eqnarray}
 with the inner product defined based on the Gaussian weight
\begin{eqnarray}
\langle h,g \rangle_{\mu_{\Xi}}&=&\int_{\mathbb{R}^{n}}h(q)g(q)f_{\Xi}(q)d q. 
\end{eqnarray}
Therefore  
\begin{eqnarray}
\frac{d m_{i,\alpha}}{dt}&=&\langle b_i,{H}_\alpha\rangle_{\mu_\Xi}-\frac{1}{2}\beta^2\int_{\mathbb{R}^{n}}{{H}_\alpha}(\xi)\left(\nabla_{x_i}  \log f_{X}(x;t)\right)_{x=M}d\mu_\Xi \nonumber \\
&=&\langle  b_i,{H}_\alpha \rangle _{\mu_\Xi}+\frac{1}{2}\beta^2\bigg \langle \left(\frac{\partial {M}_l}{\partial \xi_k}\right)^{-1},\frac{\partial {H}_\alpha}{\partial \xi_l}\bigg \rangle_{\mu_\Xi} ,\label{eq:ODE}
\end{eqnarray}
and 
\begin{eqnarray}
\frac{\partial {{M}}_i}{\partial \xi_k}\left(\frac{\partial {{M}}_j}{\partial \xi_k}\right)^{-1}&=&\delta_{ij},
\end{eqnarray}
with $\delta$ being the Kronecker delta. Note that in deriving the last step of Eq.~\eqref{eq:ODE}, the fact that $f_{\Xi}$ vanishes at the boundaries together with Eq.~\eqref{eq:consist} have been used. Moreover since $f_{X}, f_{\Xi} \in K_1$, the inverse of  $\nabla_\xi {M} $ exists which can be seen again from Eq.~\eqref{eq:consist}. 
It is important to emphasize that the evolution of the coefficients $m_{i,\alpha}$ do not directly depend on $f_X$. By taking advantage of  the measure transform \eqref{eq:consist},  no explicit knowledge of the density $f_X$ is required. \\ \ \\
In practice, basides the error associated with the regularization of the initial condition, three types of numerical errors should be controlled in order to compute the evolution of the coefficients $m_{i,\alpha}$. First type comes through truncation of the Hermite expansion \eqref{eq:exp-init}. Second is due to the inner products $\langle .,. \rangle_{\mu_\Xi} $, where the Hermite-Gauss quadrature can be employed. And third, the error arising from the time integration which can be performed e.g. by the Runge-Kutta method, should be curbed.
\section{Conclusion} 
This study proposed a transformation of the diffusion arising from the white noise into a transport induced by logarithmic gradient of the probability density. The well-posedeness of such a transformation for an It\^{o} process with strong regularity assumptions was shown. As a result, the transformed It\^{o} process behaves similar to an ODE with uncertain initial condition. Therefore the process remains measurable with respect to its initial condition resulting in interesting computational advantages. The relevance of the transformation was discussed by employing the chaos expansion technique. In follow up studies, besides analyzing the computational performance of the resulting chaos expansion, the author will investigate possible generalization of the transformation for a broader class of stochastic processes driven by the white noise.

\section*{Acknowledgement}
The author is grateful to Jan Hesthaven for his valuable comments on this study.

\bibliographystyle{siamplain}
\bibliography{references}

\end{document}